\documentclass{amsart}
\usepackage{amsmath}
\usepackage{amsfonts}
\usepackage{amssymb,enumerate}
\usepackage{amsthm}
\usepackage[all]{xy}
\usepackage{rotating}
\usepackage{hyperref}
\newcommand{\HH}{{\operatorname H}}

\newcommand{\im}{\operatorname{im}}

\newcommand{\fm}{{\mathfrak{m}}}

\newcommand{\fp}{{\mathfrak{p}}}

\newcommand{\xra}{\xrightarrow}

\theoremstyle{plain}
\newtheorem{theorem}{Theorem}[section]

\newtheorem{proposition}[theorem]{Proposition}
\newtheorem{lemma}[theorem]{Lemma}
\newtheorem{corollary}[theorem]{Corollary}

\theoremstyle{definition}
\newtheorem{definition}[theorem]{Definition}

\newtheorem{example}[theorem]{Example}

\theoremstyle{remark}
\newtheorem{remark}[theorem]{Remark}

\newtheorem{notation}[theorem]{Notation}

\newtheoremstyle{myplain}
     {1em plus 0.15em minus 0.05em}%      Space above
     {1em plus 0.15em minus 0.05em}%      Space below
     {\itshape}% Body font
     {}%         Indent amount (empty = no indent, \parindent = para indent)
     {\bf}%      Thm head font
     {.}%        Punctuation after thm head
     {0.5em}%    Space after thm head: " " = normal interword space;
     {}%         Thm head spec (can be left empty, meaning `normal')

\newtheoremstyle{mydef}
     {1em plus 0.15em minus 0.05em}%      Space above
     {1em plus 0.15em minus 0.05em}%      Space below
     {}%         Body font
     {}%         Indent amount (empty = no indent, \parindent = para indent)
     {\bf}%      Thm head font
     {.}%        Punctuation after thm head
     {0.5em}%    Space after thm head: " " = normal interword space;
     {}%         Thm head spec (can be left empty, meaning `normal')

\newtheoremstyle{myrmk}
     {1em plus 0.15em minus 0.05em}%      Space above
     {1em plus 0.15em minus 0.05em}%      Space below
     {}%         Body font
     {}%         Indent amount (empty = no indent, \parindent = para indent)
     {\itshape}% Thm head font
     {.}%        Punctuation after thm head
     {0.5em}%    Space after thm head: " " = normal interword space;
     {}%         Thm head spec (can be left empty, meaning `normal')

\theoremstyle{myplain}

\theoremstyle{myrmk}

\theoremstyle{mydef}

\newtheorem{para}[theorem]{}

\numberwithin{equation}{theorem}
\newcommand{\und}[1]{#1^{\natural}}
\newcommand{\dgspec}[1][A]{\operatorname{DGSpec}(#1)}
\newcommand{\spec}[1][R]{\operatorname{Spec}(#1)}
\newcommand{\dgdim}[1][A]{\operatorname{DGdim}(#1)}
\newcommand{\dgcontr}[2]{#1^{#2}}

\newcommand{\supp}[2][R]{\operatorname{Supp}_{#1}(#2)}
\newcommand{\anc}[2][R]{\operatorname{Anc}_{#1}(#2)}
\newcommand{\ldim}[2][R]{\operatorname{ldim}_{#1}(#2)}

\newcommand{\Lotimes}[3][R]{#2\otimes_{#1}#3}
\newcommand{\x}{\mathbf{x}}

\begin{document}

\bibliographystyle{amsplain}

\author{Kristen A. Beck}

\address{Kristen A. Beck, Department of Mathematics,
The University of Arizona,
617 N. Santa Rita Ave.,
P.O. Box 210089,
Tucson, AZ 85721,
USA}

\email{kbeck@math.arizona.edu}

\urladdr{http://math.arizona.edu/\~{}kbeck/}

\author{Sean Sather-Wagstaff}

\address{Sean Sather-Wagstaff,
Department of Mathematics,
NDSU Dept \# 2750,
PO Box 6050,
Fargo, ND 58108-6050
USA}

\email{sean.sather-wagstaff@ndsu.edu}

\urladdr{http://www.ndsu.edu/pubweb/\~{}ssatherw/}

%\thanks{This material is based on work supported by North Dakota EPSCoR and 
%National Science Foundation Grant EPS-0814442.
%Sean Sather-Wagstaff was supported in part by a grant from the NSA}

\title[Krull Dimension  for DG Algebras]
{Krull Dimension for Differential Graded Algebras}

%\date{\today}

%\dedicatory{}


\begin{abstract}
We introduce a naive notion of a system of parameters for a homologically finite complex
over a commutative noetherian local ring, and compare it to the system of parameters defined
by Christensen. We show that these notions differ in general, but that they agree when the 
complex in question is a DG $R$-algebra. In this case we also show that the Krull dimension defined in terms of
the lengths of such systems of parameters agrees with Krull dimensions defined in terms of certain
chains of prime ideals.
\end{abstract}

\maketitle

%\tableofcontents

\section{Introduction} \label{sec111025a}

In this paper, $R$ is a commutative noetherian ring with identity.
The term ``$R$-complex'' is short for ``chain complex of (unital) $R$-modules''  indexed homologically.
The \emph{infimum} of an $R$-complex $X$ is
$\inf(X):=\inf\{i\in\mathbb Z\mid\HH_i(X)\neq 0\}$,
and $X$ is \emph{homologically finite} if the total homology module $\coprod_{i\in\mathbb Z}\HH_i(X)$ is finitely generated.
The Koszul complex over $R$ on a sequence $\x=x_1,\ldots,x_n\in R$ is denoted $K^R(\x)$.

Foxby~\cite{foxby:bcfm} 
defines the \emph{Krull dimension} of an $R$-complex $X$ as
$$\dim_R(X):=\sup\{\dim(R/\fp)-\inf(X_{\fp})\mid\fp\in\supp X\}$$
where $\supp X:=\cup_{i\in\mathbb Z}\supp{\HH_i(X)}$.
If $M$ is a finitely generated $R$-module, then $\dim_R(M)$ is the usual Krull dimension of $M$,
given in terms of chains of prime ideals in $\supp M$.
If $\inf(X)>-\infty$, then $\dim_R(X)\geq-\inf(X)$.

When $R$ is local, it is natural to seek a notion of a systems of parameters for  homologically finite $R$-complexes.
One such notion comes from Christensen~\cite{christensen:sc2}, starting with the following version of minimal prime ideals for complexes.
Let $X$ be an $R$-complex such that $\inf(X)>-\infty$. A prime ideal $\fp\in\spec$ is an \emph{anchor prime} for $X$ 
if $\dim_{R_{\fp}}(X_{\fp})=-\inf(X_{\fp})$.
Let $\anc X$ denote the set of anchor primes for $X$.
Assuming that $(R,\fm)$ is local and $X$ is homologically finite, 
a \emph{system of parameters} for  $X$ is a sequence $\mathbf x=x_1,\ldots,x_d\in\fm$ 
such that $\fm\in\anc{\Lotimes{K^R(\x)}{X}}$ and $d=\dim_R(X)+\inf(X)$.
Christensen~\cite[Theorem 2.9]{christensen:sc2} shows that $X$
has a system of parameters in this setting.

The point of this paper is to explore the following different (possibly more naive) versions of these notions.

\begin{definition}\label{defn121025d}
Assume that $(R,\fm)$ is local, and let $X$ be a homologically finite $R$-complex. 
A \emph{length sequence} for $X$ is a sequence $\x=x_1,\ldots,x_d\in\fm$
such that each $\HH_i(\Lotimes{K^R(\x)}{X})$ has finite length.
If $m$ is the length of the shortest  length sequence for $X$,
then the \emph{length dimension} of $X$ is 
$$\ldim X:=m-\inf(X).$$
A \emph{length system of parameters} for $X$ is a
length sequence $x_1,\ldots,x_m\in\fm$ for $X$ such that $m=\ldim X+\inf(X)$.
\end{definition}

\begin{remark}\label{rmk121025c}
Assume that $(R,\fm)$ is local, and let $X$ be a homologically finite $R$-complex. 
Any generating sequence for an $\fm$-primary ideal of $R$ is a  length sequence for $X$,
so $X$ admits a length system of parameters, and
$$\dim(R)\geq\ldim X+\inf(X).$$
We have $\ldim X\geq-\inf(X)$, with equality holding if and only if 
each $\HH_i(X)$ has finite length.
\end{remark}

Lemma~\ref{lem121025b} shows that $\ldim X\geq\dim_R(X)$.
It is straightforward to show that one can have strict inequality here;
see Example~\ref{ex121025b}.
On the other hand, the main result of this paper shows that this cannot occur when $X$ is a DG $R$-algebra.
It is stated next and proved in~\ref{proof121026a}.
See Section~\ref{sec111025c} for background on DG algebras.

\begin{theorem}\label{thm121026a}
Let $A$ be a homologically finite positively graded commutative local noetherian DG $A_0$-algebra such that $(A_0,\fm_0)$ is local noetherian.
\begin{enumerate}[\rm(a)]
\item \label{thm121026a1}
Given a sequence $\x\in\fm_0$, the following conditions are equivalent:
\begin{enumerate}[\rm(i)]
\item $\x$ is a system of parameters for $A$;
\item $\x$ is a system of parameters for $\HH_0(A)$; and
\item $\x$ is a length system of parameters for $A$.
\end{enumerate}
\item \label{thm121026a2}
One has $\ldim[A_0]A=\dim_{A_0}(A)=\dim(\HH_0(A))$.
\item \label{thm121026a3}
If $A$ is generated over $A_0$ in odd degrees or if $A$ is bounded, then 
$\dgdim=\ldim[A_0]A=\dim_{A_0}(A)=\dim(\HH_0(A))$.
\end{enumerate}
\end{theorem}

\section{DG Algebras and DG Krull Dimension} \label{sec111025c}

We begin this section with a summary of terminology from~\cite{avramov:dgha,felix:rht}.

\begin{notation}
Given an $R$-complex $X$, write $|x|=i$ when $x\in X_i$.
\end{notation}

\begin{definition}
\label{DGK}
A \emph{positively graded commutative differential graded $R$-algebra} (\emph{DG $R$-algebra} for short)
is an $R$-complex $A$ equipped with a
%chain map 
%$\mu^A\colon A\otimes_RA\to A$ with $ab:=\mu^A(a\otimes b)$
%that is:
binary operation $(a,b)\mapsto ab$ 
satisfying the following properties:
\begin{description}
\item[associative] for all $a,b,c\in A$ we have $(ab)c=a(bc)$;
\item[distributive] for all $a,b,c\in A$ such that 
$|a|=|b|$ we have $(a+b)c=ac+bc$ and $c(a+b)=ca+cb)$;
\item[unital] there is an element $1\in A_0$ such that for all $a\in A$ we have $1a=a$;
\item[graded commutative] for all $a,b\in A$ we have 
$ba = (-1)^{|a||b|}ab\in A_{|a|+|b|}$, and $a^2=0$ when
$|a|$ is odd; 
\item[positively graded] $A_i=0$ for $i<0$; and
\item[Leibniz Rule] 
for all $a,b\in A$ we have 
$\partial^A_{|a|+|b|}(ab)=\partial^A_{|a|}(a)b+(-1)^{|a|}a\partial^A_{|b|}(b)$.
\end{description}
%The map $\mu^A$ is the \emph{product} on $A$.
Given a DG $R$-algebra $A$, the \emph{underlying algebra} is the
graded commutative  $R$-algebra
$\und{A}=\oplus_{i=0}^\infty A_i$.

We say that $A$ is \emph{noetherian}
if $\HH_0(A)$ is noetherian and the $\HH_0(A)$-module $\HH_i(A)$ is  finitely generated for all $i\geq 0$.
We say that $A$ is  \emph{local} if it is noetherian, $R$ is local, and
the ring $\HH_0(A)$ is a local $R$-algebra\footnote{This means that
$\HH_0(A)$ is a local  ring whose maximal ideal contains the ideal $\fm\HH_0(A)$ where $\fm$ is the maximal 
ideal of $R$.}.
\end{definition}

\begin{example}
Given a sequence $\x=x_1,\ldots,x_n\in R$, the Koszul complex $K^R(\x)$ is a noetherian DG $R$-algebra
under the wedge product;  it is generated over $K_0=R$ by $K_1$. If $(R,\fm)$ is local and $\x\in\fm$, then $K^R(\x)$ is a local DG $R$-algebra.
\end{example}

\begin{definition}
Let $A$ be a DG $R$-algebra. A \emph{differential graded module over}~$A$
(\emph{DG $A$-module} for short) is an $R$-complex $M$ equipped with a
binary operation $(a,m)\mapsto am$ 
satisfying the following properties:
\begin{description}
\item[associative] for all $a,b\in A$ and $m\in M$ we have $(ab)m=a(bm)$;
\item[distributive] for all $a,b\in A$ and $m,n\in M$ such that 
$|a|=|b|$ and $|m|=|n|$, we have $(a+b)m=am+bm$ and $a(m+n)=am+an)$;
\item[unital] for all $m\in M$ we have $1m=m$;
\item[graded] for all $a\in A$ and $m\in M$ we have 
$am\in M_{|a|+|m|}$; 
\item[Leibniz Rule] 
for all $a\in A$ and $m\in M$ we have 
$\partial^A_{|a|+|m|}(am)=\partial^A_{|a|}(a)m+(-1)^{|a|}a\partial^M_{|m|}(m)$.
\end{description}
The \emph{underlying $\und{A}$-module} associated to $M$ is the
$\und{A}$-module
$\und{M}=\oplus_{i=-\infty}^\infty M_i$.

A \emph{DG submodule} of a DG $A$-module $M$ is a subcomplex that is a DG $A$-module under the 
operations induced from $M$. A \emph{DG ideal} of $A$ is a DG submodule of $A$.
\end{definition}

\begin{definition}\label{defn111025b}
Let $A$ be a DG $R$-algebra. A DG ideal $I\subseteq A$ is \emph{prime} if
$\und I$ is a prime ideal of $\und A$.
Let $\dgspec$ denote the set of DG prime ideals of $A$.
The \emph{DG Krull dimension} of $A$, denoted $\dgdim$,
is the supremum of lengths of chains of DG prime ideals of $A$.
For each ideal $I\subseteq \HH_0(A)$, 
write $I=\tilde I/\im(\partial^A_1)$ where $\tilde I$ is an ideal of $A_0$ containing $\im(\partial^A_1)$,
and set 
$$\dgcontr IA\cdots\xra{\partial^2_A}A_1\xra{\partial^1_A}\tilde I\to 0.$$
\end{definition}

\begin{remark}\label{rmk111025a}
Let $A$ be a  DG $R$-algebra.
The following facts are straightforward to verify.
For each DG (prime) ideal $J\subseteq A$, the subset $J_0\subseteq A_0$ is a
(prime) ideal
containing $\partial^A_1(J_1)$.
For each  ideal $I\subseteq \HH_0(A)$, the subset $\dgcontr IA\subseteq A$ is a DG 
ideal of $A$. 
An ideal $I\subseteq \HH_0(A)$ is prime if and only if $\dgcontr IA\subseteq A$ is  DG prime.
The operation $I\mapsto \dgcontr IA$, 
considered as a map from the set of (prime) ideals of $\HH_0(A)$ to the set of DG
(prime) ideals of $A$,
is injective and
respects containments.
In particular, one has $\dgdim\geq\dim(\HH_0(A))$.
\end{remark}

\begin{proposition}\label{prop111025a}
Let $A$ be a  DG $R$-algebra.
If $A$ is generated over $A_0$ in odd degrees or if $A$ is bounded, then 
the map $\dgcontr {(-)}A\colon\spec[\HH_0(A)]\to\dgspec$ is bijective, so 
$\dgdim=\dim(\HH_0(A))$.
\end{proposition}

\begin{proof}
Assume that $A$ is generated over $A_0$ in odd degrees.
Since each element $a\in A$ of odd degree is square-zero, it must be contained
in each DG prime ideal of $A$. That is, each DG prime ideal  $P\subset A$ contains
$A_+=\cdots\to A_1\to 0$.
Since $P$ must be closed under $\partial^A_1$,
it must contain 
$\dgcontr 0A=\cdots\to A_1\to \im(\partial^1_A)\to 0$.
From this it follows that $P=\dgcontr {(P_0/\im(\partial^1_A))}A$. 
Since $P$ is a DG prime of $A$, Remark~\ref{rmk111025a} implies that $P_0/\im(\partial^1_A)$ is a prime ideal
of $\HH_0(A)=A/\im(\partial^1_A))$. Thus,
the map $\dgcontr {(-)}A\colon\spec[\HH_0(A)]\to\dgspec$ is surjective, hence it is bijective
by Remark~\ref{rmk111025a}. The equality $\dgdim=\dim(\HH_0(A))$ follows immediately.

In the case where $A$ is bounded, it follows that every element $a\in A$ of non-zero degree is nilpotent,
so the above argument applies.
\end{proof}

\begin{corollary}\label{cor121024a}
Let $K=K^R(\mathbf x)$ be a Koszul complex over $R$.
Then 
the map $\dgcontr {(-)}K\colon\spec[R/(\mathbf x)]\to\dgspec[K]$ is bijective, so 
$\dgdim[K]=\dim(R/(\mathbf x))$.
\end{corollary}

The following example shows that the assumptions on $A$
(generated in odd degrees or bounded) are necessary in 
Proposition~\ref{prop111025a}.

\begin{example}
Let $k$ be a field, and let $A=k[X]$ denote the polynomial ring in one
indeterminate $X$ of degree 2. This is a  DG $k$-algebra, using the trivial
differential. The ideals $0$ and $A_+=(X)A$ are DG prime.
(Moreover, $\dgspec$ is precisely the set of graded prime ideals of $A$.)
In particular, we have $\dgdim=1>0=\dim(k)=\dim(\HH_0(A))$
since $\HH_0(A)=k$.
\end{example}

For our main theorem, we require some DG localization.

\begin{definition}\label{defn111025c}
Let $A$ be a DG $R$-algebra.
A subset $U\subseteq A$ is \emph{multiplicatively closed}
if it contains 1 and is closed under multiplication. 
Given a DG $A$-module $M$  (e.g., $M=A$) and a multiplicatively closed subset
$U\subseteq A$, we define an  relation
on $M\times U$ as follows: $(m,u)\sim(n,v)$ if 
$|m|-|u|=|n|-|v|$ and
there is an element $w\in U$
such that $w(un-(-1)^{|u||v|}vm)=0$.
\end{definition}

\begin{proposition}\label{prop111025b}
Let $A$ be a DG $R$-algebra.
Given a DG $A$-module $M$ (e.g., $M=A$)
and a multiplicatively closed subset
$U\subseteq A$, the relation
from Definition~\ref{defn111025c} is an equivalence relation.
\end{proposition}

\begin{proof}
Symmetry and transitivity are
tedious but straightforward to verify.
There is a tiny subtlety with reflexivity. To check that 
$(m,u)\sim(m,u)$, we need to consider two cases.
The case where $|u|$ is even is straightforward.
For the case where $|u|$ is odd, it follows that $u^2=0$,
so we have
$u(um-(-1)^{|u||u|}um)=0$.
\end{proof}

\begin{definition}\label{defn111025d}
Let $A$ be a DG $R$-algebra, and let
$U\subseteq A$ be multiplicatively closed.
Let $M$ be a DG $A$-module  (e.g., $M=A$).
For each $(m,u)\in M\times U$, let $m/u$ and $\frac{m}{u}$ denote the equivalence class of $(m,u)$
under the equivalence relation $\sim$
from Definition~\ref{defn111025c}. 

We define the 
\emph{DG localization} $U^{-1}M$ using the quotient rule:
\begin{align*}
(U^{-1}M)_i&:=
\{m/u\mid i=|m|-|u|\} \\
\partial^{U^{-1}M}\left(\frac{m}{u}\right)
&:=\frac{u\partial^M(m)-\partial^A(u)m}{u^2}
\\
\frac{m}{u}+\frac{m'}{u'}
&:=\frac{um'+(-1)^{|u||u'|}u'm}{uu'}\\
\frac{a}{u}\frac{m}{v}
&:=\frac{am}{uv}
\end{align*}
\end{definition}

\begin{proposition}\label{prop111025c}
Let $A$ be a DG $R$-algebra, and let
$U\subseteq A$ be multiplicatively closed.
Let $M$ be a DG $A$-module  (e.g., $M=A$).
\begin{enumerate}[\rm(a)]
\item\label{prop111025c1}
Using the above definition, $U^{-1}A$ is a DG $R$-algebra, not necessarily positively graded, and $U^{-1}M$ is a DG $U^{-1}A$-module.
\item\label{prop111025c2}
If $U\subseteq A_0$, then $U^{-1}A$ is positively graded.
\end{enumerate}
\end{proposition}

\begin{proof}
Note that if $U$ contains an element $u$ of odd degree, then everything is trivial:
the fact that $u$ has odd degree implies that $u^2=0$, so for all $m/v\in U^{-1}M$ we have 
$m/v=(u^2m)/(u^2v)=0$.
Thus, for the remainder of this proof, we assume that $U$ does not contain any elements of odd degree.
It is straightforward to show that the addition and multiplication rules for $U^{-1}A$ and $U^{-1}M$ are well-defined
and satisfy the standard axioms (associative, etc.).

We  show that the differential $\partial^{U^{-1}M}$ is well-defined. (The special case $M=A$ then follows.)
To this end, let $m/u=n/v$ in $U^{-1}M$. Since $|u|$ and $|v|$ are even, it follows that there is an element
$w\in U$ such that $w(vm-un)=0$. 
Applying $\partial^M$ to this equation, we have the first equality in the next display:
\begin{align*}
0
&=\partial^M(w(vm-un))\\
&=\partial^A(w)(vm-un)+w\partial^M(vm-un)\\
&=\partial^A(w)(vm-un)+w\partial^A(v)m+vw\partial^M(m)-w\partial^A(u)n-uw\partial^M(n).
\end{align*}
The second and third equalities follow from the Leibniz rule, since $|u|$, $|v|$, and $|w|$ are even.
The fact that $w(vm-un)=0$ implies that $w\partial^A(w)(vm-un)$.
Thus, if we multiply the above display by $uvw$, we have the first equality in the next display:
\begin{align*}
0
&=uvw^2\partial^A(v)m+uv^2w^2\partial^M(m)-uvw^2\partial^A(u)n-u^2vw^2\partial^M(n)\\
&=uw\partial^A(v)(wvm)+uv^2w^2\partial^M(m)-vw\partial^A(u)(wun)-u^2vw^2\partial^M(n)\\
&=uw\partial^A(v)(wun)+uv^2w^2\partial^M(m)-vw\partial^A(u)(wvm)-u^2vw^2\partial^M(n)\\
&=w^2(u^2\partial^A(v)n+uv^2\partial^M(m)-v^2\partial^A(u)m-u^2v\partial^M(n)).
\end{align*}
The second and fourth equalities follow from the fact that $|u|$, $|v|$, and $|w|$ are even.
The third equality follows from the condition $w(vm-un)=0$.
This explains the second equality in the next display
\begin{align*}
\frac{u\partial^M(m)-\partial^A(u)m}{u^2}
&=\frac{uv^2\partial^M(m)-v^2\partial^A(u)m}{u^2v^2}\\
&=\frac{u^2v\partial^M(n)-u^2\partial^A(v)n}{u^2v^2}\\
&=\frac{v\partial^M(n)-\partial^A(v)n}{v^2}
\end{align*}
so we conclude that $\partial^{U^{-1}M}$ is well-defined.

Next, we show that $\partial^{U^{-1}M}\partial^{U^{-1}M}=0$.  (The special case $M=A$ then follows.)
\begin{align*}
\partial^{U^{-1}M}\left(\partial^{U^{-1}M}\left(\frac{m}{u}\right)\right)\\
&\hspace{-1.5cm}=\partial^{U^{-1}M}\left(\frac{u\partial^M(m)-\partial^A(u)m}{u^2}\right)\\
&\hspace{-1.5cm}=\frac{u^2\partial^M(u\partial^M(m)-\partial^A(u)m)-\partial^A(u^2)(u\partial^M(m)-\partial^A(u)m)}{u^4}\\
&\hspace{-1.5cm}=\frac{u^2(\partial^A(u)\partial^M(m)+u\partial^M(\partial^M(m))-\partial^A(\partial^A(u))m+\partial^A(u)\partial^M(m))}{u^4}\\
&\hspace{-1cm}-\frac{2u\partial^A(u)(u\partial^M(m)-\partial^A(u)m)}{u^4} \\
&\hspace{-1.5cm}=\frac{u^2(2\partial^A(u)\partial^M(m))-2u\partial^A(u)(u\partial^M(m)-\partial^A(u)m)}{u^4}\\
&\hspace{-1.5cm}=\frac{2u\partial^A(u)\partial^A(u)m}{u^4}\\
&\hspace{-1.5cm}=0
\end{align*}
The first two steps are by definition.
The third step follows from the Leibniz rule for $M$,
and the fourth step uses the fact that $\partial^M\partial^M=0=\partial^A\partial^A$.
The fifth step is straightforward cancellation.
For the sixth step, note that the fact that $|u|$ is even implies that $|\partial^A(u)|$ is odd, so the element
$\partial^A(u)\in A$ is square-zero.

The Leibniz rule for $U^{-1}M$ (and hence for $U^{-1}A$) is straightforward.
\end{proof}

\section{Dimension and  Systems of Parameters} \label{sec121025a}

Before proving Theorem~\ref{thm121026a}, we require a few more preliminaries.

\begin{lemma}\label{lem121025a}
Let $X$ be a homologically bounded below $R$-complex, and let $\fm\subset R$ be a maximal ideal. 
If $\supp X=\{\fm\}$, e.g., if $X\not\simeq 0$ and each $\HH_i(X)$ has finite length over $R$, then $\fm\in\anc X$.
\end{lemma}

\begin{proof}
By definition, we have
\begin{align*}
\dim_R(X)
&=\sup\{\dim(R/\fp)-\inf(X_{\fp})\mid\fp\in\supp X\}
\\
&=\dim(R/\fm)-\inf(X_{\fm})\\
&=-\inf(X_{\fm})
\end{align*}
as desired.
\end{proof}

Here is an example showing that the converse of the previous result fails.

\begin{example}\label{ex121025a}
Let $k$ be a field and set $R=k[\![T]\!]$ with $\fm=TR$.
Consider the following complex, which is concentrated in degrees 0 and 1:
$$X\quad =0\to R\xrightarrow 0 k\to 0.$$
Since $\HH_1(X)\cong R$, we have $\supp X=\spec$.
And we compute:
\begin{align*}
\dim_R(X)
&=\sup\{\dim(R/\fp)-\inf(X_{\fp})\mid\fp\in\supp X\}
\\
&=\sup\{\dim(R/\fm)-\inf(X_{\fm}),\dim(R/(0))-\inf(X_{(0)})\} \\
&=\sup\{-\inf(X),1-\inf(X_{(0)})\} \\
&=\sup\{0,1-1\} \\
&=0\\
&=-\inf(X).
\end{align*}
So, we have $\fm\in\anc X$ and $\supp X\neq\{\fm\}$.
\end{example}

\begin{lemma}\label{lem121025b}
Assume that $(R,\fm)$ is local, and fix  a homologically finite $R$-complex $X$. 
Each length system of parameters $\x$ for $X$ satisfies $\fm\in\anc{\Lotimes{K^R(\x)}{X}}$.
In particular, we have $\ldim X\geq\dim_R(X)$.
\end{lemma}

\begin{proof}
Let $\x=x_1,\ldots,x_m\in\fm$ be a length system of parameters for $X$. By definition of $\dim_R(X)$,
it suffices to show that $\fm\in\anc{\Lotimes{K^R(\x)}{X}}$.
Since $\x$ is a length system of parameters for $X$, we know that
each $\HH_i(\Lotimes{K^R(\x)}{X})$ has finite length, so we have $\fm\in\anc{\Lotimes{K^R(\x)}{X}}$
by Lemma~\ref{lem121025a}.
\end{proof}

Example~\ref{ex121025a} shows that equality can fail in the previous result, as we see next.

\begin{example}\label{ex121025b}
Let $k$ be a field and set $R=k[\![T]\!]$ with $\fm=TR$.
Consider the following complex, which is concentrated in degrees 0 and 1:
$$X\quad =\quad0\to R\xrightarrow 0 k\to 0.$$
We have already seen that $\dim_R(X)=0$. Since $\HH_1(X)\cong R$ does not have finite length, we have
$\ldim X>0=\dim_R(X)$.
(More specifically, it is straightforward to show that $\ldim X=1$.)
\end{example}

Theorem~\ref{thm121026a} from the introduction follows from the next result; see~\ref{proof121026a}.

\begin{proposition}\label{thm121025a}
Let $A$ be a homologically finite positively graded commutative local noetherian DG $A_0$-algebra such that $(A_0,\fm_0)$ is local noetherian.
\begin{enumerate}[\rm(a)]
\item \label{thm121025a1}
Given a system of parameters $\x\in\fm_0$ for $\HH_0(A)$, each $\HH_i(\Lotimes[A_0]{K^{A_0}(\x)}{A})$ has finite length over $A_0$.
In particular, we have $\dim(\HH_0(A))\geq\ldim[A_0]A$.
\item \label{thm121025a4}
Given a system of parameters $\x\in\fm_0$ for $A$, the ring $\HH_0(A)/(\x)\HH_0(A)$ is artinian.
In particular, we have $\dim_{A_0}(A)\geq\dim(\HH_0(A))$.
\end{enumerate}
\end{proposition}

\begin{proof}
\eqref{thm121025a1}
Let $\x\in\fm_0$ be a system of parameters  for $\HH_0(A)$. It follows that
$\Lotimes[A_0]{K^{A_0}(\x)}{A}$ is a homologically finite local noetherian DG $A_0$-algebra such that $(A_0,\fm_0)$ is local noetherian.
Furthermore, the ring 
$$\HH_0(K^{A_0}(\x)\otimes_{A_0}A)\cong \HH_0(A)/(\x)\HH_0(A)$$ 
is local and artinian.
Since each $\HH_i(K^{A_0}(\x)\otimes_{A_0}A)$ is finitely generated over $\HH_0(K^{A_0}(\x)\otimes_{A_0}A)$, it follows that each
$\HH_i(K^{A_0}(\x)\otimes_{A_0}A)$ has finite length.

\eqref{thm121025a4}
Let $\x\in\fm_0$ be a system of parameters  for $A$. 
By definition, this implies that $\fm_0\in\anc[A_0]{\Lotimes[A_0]{K^{A_0}(\x)}{A}}$.
This explains the second equality in the next display:
\begin{align*}
0
&=-\inf(\Lotimes[A_0]{K^{A_0}(\x)}{A}) \\
&= \dim_{A_0}(\Lotimes[A_0]{K^{A_0}(\x)}{A}) \\
&=\sup\{\dim(A_0/\fp_0)-\inf((\Lotimes[A_0]{K^{A_0}(\x)}{A})_{\fp_0}) \mid\fp_0\in\supp[A_0]{\Lotimes[A_0]{K^{A_0}(\x)}{A}}\}
\end{align*}
The first equality is by 
the isomorphism $\HH_0(\Lotimes[A_0]{K^{A_0}(\x)}{A}) \cong \HH_0(A)/(\x)\HH_0(A)$ and Nakayama's Lemma, and the third one is by definition.

Claim: We have 
$$\supp[A_0]{\Lotimes[A_0]{K^{A_0}(\x)}{A}}=\supp[A_0]{\HH_0(\Lotimes[A_0]{K^{A_0}(\x)}{A})}=\supp[A_0]{\HH_0(A)/(\x)\HH_0(A)}$$
and 
for each $\fp_0\in\supp[A_0]{\Lotimes[A_0]{K^{A_0}(\x)}{A}}$, we have $\inf((\Lotimes[A_0]{K^{A_0}(\x)}{A})_{\fp_0})=0$.
The equality $\supp[A_0]{\HH_0(\Lotimes[A_0]{K^{A_0}(\x)}{A})}=\supp[A_0]{\HH_0(A)/(\x)\HH_0(A)}$
follows from the isomorphism $\HH_0(\Lotimes[A_0]{K^{A_0}(\x)}{A})\cong \HH_0(A)/(\x)\HH_0(A)$.
And the containment $\supp[A_0]{\Lotimes[A_0]{K^{A_0}(\x)}{A}}\supseteq\supp[A_0]{\HH_0(\Lotimes[A_0]{K^{A_0}(\x)}{A})}$
is a consequence of the defi-nition $\supp[A_0]{\Lotimes[A_0]{K^{A_0}(\x)}{A}}=\cup_{i}\supp[A_0]{\HH_i(\Lotimes[A_0]{K^{A_0}(\x)}{A})}$.
Now, fix a prime $\fp_0\in\supp[A_0]{\Lotimes[A_0]{K^{A_0}(\x)}{A}}$, and suppose that
$\fp_0\notin\supp[A_0]{\HH_0(\Lotimes[A_0]{K^{A_0}(\x)}{A})}$.
It follows that we have
$$0\not\simeq(\Lotimes[A_0]{K^{A_0}(\x)}{A})_{\fp_0}\simeq \Lotimes[(A_0)_{\fp_0}]{K^{(A_0)_{\fp_0}}(\x)}{A_{\fp_0}}.$$
Note that $\Lotimes[(A_0)_{\fp_0}]{K^{(A_0)_{\fp_0}}(\x)}{A_{\fp_0}}$ is a positively graded
DG $(A_0)_{\fp_0}$-algebra by Proposition~\ref{prop111025c}.
So each homology module $\HH_i(\Lotimes[(A_0)_{\fp_0}]{K^{(A_0)_{\fp_0}}(\x)}{A_{\fp_0}})$ is a
module over $\HH_0(\Lotimes[(A_0)_{\fp_0}]{K^{(A_0)_{\fp_0}}(\x)}{A_{\fp_0}})$.
The condition $\fp_0\notin\supp[A_0]{\HH_0(\Lotimes[A_0]{K^{A_0}(\x)}{A})}$ implies
that $\HH_0(\Lotimes[(A_0)_{\fp_0}]{K^{(A_0)_{\fp_0}}(\x)}{A_{\fp_0}})=0$,
so each module over this ring is 0. Hence, for all $i$ we have 
$\HH_i(\Lotimes[(A_0)_{\fp_0}]{K^{(A_0)_{\fp_0}}(\x)}{A_{\fp_0}})=0$, contradicting the non-triviality condition 
$\Lotimes[(A_0)_{\fp_0}]{K^{(A_0)_{\fp_0}}(\x)}{A_{\fp_0}}\not\simeq 0$.
The claim now follows.

Combining the claim with the previous paragraph, we have
$$0=\sup\{\dim(A_0/\fp_0) \mid\fp_0\in\supp[A_0]{\HH_0(A)/(\x)\HH_0(A)}.$$
Thus, the only prime in $\supp[A_0]{\HH_0(A)/(\x)\HH_0(A)}$ is $\fm_0$.
Since $\HH_0(A)/(\x)\HH_0(A)$ is noetherian, it follows that $\HH_0(A)/(\x)\HH_0(A)$ is artinian, as desired.
\end{proof}

\begin{para}[Proof of Theorem~\ref{thm121026a}] \label{proof121026a}
Parts~\eqref{thm121026a1} and~\eqref{thm121026a2}
follow from Proposition~\ref{thm121025a}
and Lemma~\ref{lem121025b}.
And part~\eqref{thm121026a3}
is from Proposition~\ref{prop111025a}.
\qed
\end{para}

\section*{Acknowledgments}
We are grateful to the referee for his/her careful reading of  this paper.

%\bibliography{../+new}
\providecommand{\bysame}{\leavevmode\hbox to3em{\hrulefill}\thinspace}
\providecommand{\MR}{\relax\ifhmode\unskip\space\fi MR }
\providecommand{\MRhref}[2]{%
  \href{http://www.ams.org/mathscinet-getitem?mr=#1}{#2}
}
\providecommand{\href}[2]{#2}

\end{document}